\documentclass[10pt,twocolumn,twoside]{IEEEtran}
\usepackage{amsthm}
\usepackage{amsmath}
\usepackage{graphicx}
\usepackage{cite}
\usepackage{amssymb}
\usepackage{color}
\usepackage{algorithm}
\usepackage{algorithmic}

\newtheorem{assumption}{Assumption}
\newtheorem{observation}{Observation}
\newtheorem{lemma}{Lemma}
\newtheorem{definition}{Definition}
\newtheorem{theorem}{Theorem}
\newtheorem{remark}{Remark}
\newtheorem{corollary}{Corollary}

\usepackage{booktabs,lipsum}
\usepackage{mathtools}
\usepackage{mathptmx}

\ifCLASSINFOpdf

\else

\fi

\begin{document}

\title{A Differential Game Approach to Decentralized Virus-Resistant Weight Adaptation Policy over Complex Networks}
%
%
%

\author{Yunhan~Huang,~\IEEEmembership{}
        Quanyan~Zhu,~\IEEEmembership{Member,~IEEE}
\thanks{Y. Huang and Q. Zhu are both with the Department of Electrical and Computer Engineering, New York University, Brookly, NY, USA; e-mail: \{yh2315,qz494\}@nyu.edu.}
}

\maketitle
\begin{abstract}
Increasing connectivity of communication networks enables large-scale distributed processing over networks and improves the efficiency for information exchange. However, malware and virus can take advantage of the high connectivity to spread over the network and take control of devices and servers for illicit purposes. In this paper, we use an SIS epidemic model to capture the virus spreading process and develop a virus-resistant weight adaptation scheme to mitigate the spreading over the network. We propose a differential game framework to provide a theoretic underpinning for decentralized mitigation in which nodes of the network cannot fully coordinate, and each node determines its own control policy based on local interactions with neighboring nodes. We characterize and examine the structure of the Nash equilibrium, and discuss the inefficiency of the Nash equilibrium in terms of minimizing the total cost of the whole network. A mechanism design through a penalty scheme is proposed to reduce the inefficiency of the Nash equilibrium and allow the decentralized policy to achieve social welfare for the whole network. We corroborate our results using numerical experiments and show that virus-resistance can be achieved by a distributed weight adaptation scheme.
\end{abstract}


\IEEEpeerreviewmaketitle

\section{Introduction}\label{Introduction}
\IEEEPARstart{T}{he} integration of the information and communications technologies into systems upgrades system performance. However, the integration also degrades the security level of the systems and introduces vulnerabilities that undermine the reliability of critical infrastructure. The connectivity and interdependence of cyber networks make the system even more vulnerable due to the existence of the wide-spreading cyber-attacks on networks. It provides opportunities for the sophisticated and stealthy malware and virus to spread over the network. One noteworthy example is the StuxNet attack \cite{Farwell2011}. In June 2010, certain control systems of a nuclear-enrichment plant in Iran were infected by a carefully crafted computer worm called StuxNet. The worm, spreading through USB devices, intended to breach the implemented cyberprotection schemes and alter both the measurement and actuation signals which caused instabilities and damage the physical plant\cite{Pasqualetti2015}. More recent examples of wide-spreading cyber-attacks include WannaCry and Petya Ransomware, which have incurred billions of dollars of losses \cite{Hayel2017}.

\begin{figure}\centering
\vspace{-9mm}
\includegraphics[width=0.45\textwidth]{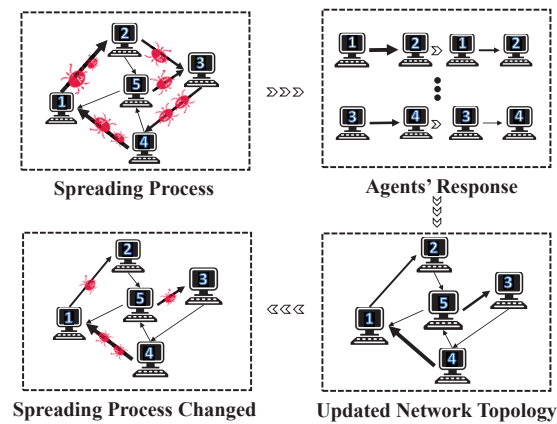}
\caption{The interactions between the spreading, topology and individual behavior. The agents adapt their behavior of contacting other agents due to virus spreading. The agents behavior changes the topology of the network which will in in turn affects the spreading process. For example, after the topology change, the virus is less likely to spread from device $1$ to device $2$.} \label{TriInter}
\end{figure}

With an increasing number of wide-spreading cyber-attacks on networks, protection against malware and virus spreading in cyber networks is central to the security of network systems \cite{Hayel2017}. However, there are many challenges on designing a protection scheme for cyber networks. One challenge is due to the interderdependency between the microscopic individual behaviors and the macroscopic spreading phenomenon.  The local interactions over a large network where nodes communicate, share information, and make interdependent decisions, can result in a macroscopic behavior, which will in turn affect the agents' behaviors. This type of microscopic and macroscopic couplings has been illustrated in Fig. \ref{TriInter}. Another challenge arises from the fact that cyber networks are often formed by a large number of self-interested agents or decision-makers. The noncooperation among the agents makes it almost impossible for the system to be coordinated as a whole to defend against wide-spreading cyber-attacks.

To this end, one way to mitigate the malware spreading over large networks is to control the intensity of interactions with neighboring nodes. By adapting the rate of communications or contacts, nodes can reduce the likelihood of infection. This type of mechanism is called weight adaptation as the weights between two nodes of a network capture the intensity of the connectivity \cite{Guo2013}. The most fundamental reason that virus and malware can go viral is the inherent property of networks: connectivity. Weight adaptation is a mechanism that hits the nail. Weight adaptation lowers the connectivity which leaves virus and malware no way out. Compared with quarantining and link removal \cite{Khouzani2012}, weight adaptation does not need to completely disconnect nodes from others but rather adjust weights to connect more loosely with nodes with a higher likelihood of infection. Instead of fixing the weights for the whole spreading process, in the weight adapation scheme, each agent dynamically updates their weight in response to the state of the neighboring nodes. {Weight adaptation is different from changing the infection rate. The infection rate is usually considered to be decided by some interior factors like physiological or immunological states of individual. The weight between two nodes is usually used to describe how strongly two nodes are connected. Changing the weight can be interpreted as an exterior change.}  


We consider a directed weighted network where the nodes and the edges represent the agents and the connections between the agents respectively. The directed connection between two nodes can be considered as one agent acquiring information/data/packet from another agent. The weight between two agents quantifies the frequency or the volume of communication between two agents \cite{Quanyan2012Comm}. The original weight is pre-designed by multilateral agreement among agents to achieve certain goals or to optimize the system performance when there is no infection. For example, in distributed estimation or learning problems over networks such as \cite{Mai2016,Zhang2017,Quanyan2012Comm}, one agent needs to communicate with its neighboring agents at a sufficient rate to find the global estimate of the state. The optimal weighting on the edges quantitatively captures the minimum required frequency of contacting neighboring nodes. As illustrated in Fig. \ref{AdapScheme}, when there are wide-spreading virus or cyber attacks, the agent can decrease the likelihood of being infected by reducing their weight with infected neighbors. The agent then restores the connections when the infected neighbors are recovered. Deviation of the weights from the optimal ones introduces cost induced by performance degradation and system inefficiency. 
Infected agents may not function normally. The agents and the network system will suffer losses. Thus, it is essential to consider the trade-off between malfunction cost caused by infection and inefficiency or performance degradation cost caused by weight deviation.

\begin{figure}\centering
\includegraphics[width=0.48\textwidth]{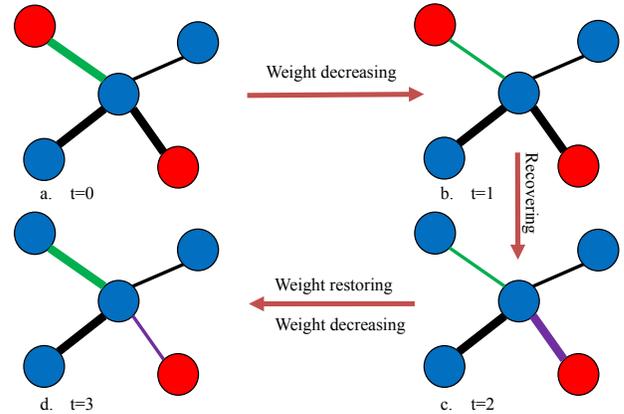}
\caption{Weight adaptation scheme to mitigate the infection for the node in the center over an undirected network. The line width indicates the weight. Red nodes are the infected ones while blue nodes are the susceptible ones. The susceptible node decreases the weight of its connection to an infected node. {Once the infected node is recovered, the weight of the connection to the recovered node will be restored as it is shown by process $c$ to $d$. The purple link and the green link are the ones who adapt the weights.}} \label{AdapScheme}
\end{figure}

In this paper, an $N$-person nonzero-sum differential game-based model is proposed to model the virus spreading and the agents' adaptive response to virus infection. This model captures the non-cooperative behaviors among agents, dynamic properties of spreading process, and the complexity of the local interactions. We characterize the Nash equilibrium (NE) for the game and investigate the network effects under the non-cooperative strategies. We observe that under the open-loop NE, each agent updates his weight based on its own infection level and its out-neighbors infection level as well as the corresponding component of its costate. When the agent's own infection level is high, it does not care much about the weight of links to infected out-neighbors. When its out-neighbor's infection level is high, it lowers more weight of the corresponding connection. The corresponding component of each agent's costate encodes the information about the network structure and the infection of the whole network.

We use a centralized optimal control problem to serve as a benchmark problem to study the efficiency of the decentralized problem. Under centralized policies, the system operator develops optimal weight adaptation scheme to achieve social optimum. Compared with the centralized solution, the open-loop NE solution is not the best from a system point of view since in the game, agents consider only their own cost. Such inefficiency caused by selfish behavior of agents has a significant impact on network and service management. One example is the congestion in traffic network caused by selfish drivers \cite{Xuehe2016}. To address the inefficiency, we propose a dynamic penalty approach by designing a mechanism in which each agent pays for the infection cost of all agents that are reachable to him/her. We show that with this mechanism, the open-loop NE policy achieves the social optimum. 

The equilibrium analysis and the mechanism design lead to a distributed algorithm for the network operator and the agents to compute the optimal weight adaptation where each agent only has to know local information. 
We summarize the principal contributions as follows:
\begin{enumerate}
\item We propose a differential game model to develop a virus-resistant weight adaptation scheme for cyber networks formed by a group of self-interested agents.
\item {We study the structure of the open-loop NE for the differential game over complex networks} and show the weight adaptation rule is based on the agents' and its out-neighbors' infection level as well as the costate.
\item We discuss the inefficiency of the NE. A dynamic penalty scheme is proposed to achieve social optimum for the whole network.
\item An implementable distributed virus resistance algorithm is proposed to compute the NE-based control policy.
\end{enumerate}

Game theory has long been a useful tool to design strategies on network systems for virus resistance purposes \cite{Hayel2017,Trajanovski2017,Hayel2017TIFS,Hayel2014}. In \cite{Trajanovski2017}, the authors have proposed a network formation game that balances multiple partially conflicting objectives such as the cost of installing links, the performance of the network and the resistance to virus. In their work, an undirected unweighted static network is formed. Hayel et.al. in \cite{Hayel2017,Hayel2017TIFS} have studied large population game with heterogeneous types of individuals. They focus on group behavior of certain type in stead of individual behavior. Besides game theory, other tools such as impulse control\cite{Shulgin1998}, optimal control\cite{Ping2018}, and optimization \cite{Preciado2014} have been used to design strategies to mitigate malware attacks and virus spreading.

Virus spreading over adaptive networks has first been studied by Gross et.al. in \cite{Gross2006}. They investigated adaptive behavior in a homogeneous way where the whole network takes the same adaption. Based on the work on epidemic spreading over time-varying networks \cite{Pare2017}, optimal control method has been utilized to find the optimal time-varying topology response for the network system in \cite{Ping2018}. However, the centralized optimal control method is not practical and lack of incentive. The effect of heterogeneous weight adaptation on virus spreading has been studied by Yun et al. in \cite{Yun2016,Yunhan2016}. In \cite{Yun2016}, the authors have proposed a weight adaptation rule without taking cost into consideration. {The weight adaptation rule is based on the infection level of the whole network.}

Vaccination and immunity have been studying for control of virus spreading over decades \cite{Shulgin1998,Preciado2014,Yunhan2016PA}. But vaccination may not be efficient for some malware and virus due to their fast upgrading and undetectability. Also, {getting every individual vaccinated is costly.} Quarantining \cite{Khouzani2012} is equivalent to removal of all connections of one agent. Compared with weight adaptation scheme, it is overreacting to disconnect all links since connection with healthy agents cause no harm. 

The paper is organized as follows. In Sect. \ref{MoForm}, preliminaries are given and the $N$-person nonzero-sum differential game framework is introduced. Section \ref{AnaRes} describes the open-loop NE of the differential game and the weight adaptation scheme. Sect. \ref{EffCom} studies the efficiency of the NE solution. Comparisons of the differential game-based weight adaptation scheme with the optimal control based scheme and other numerical results are given in Sect. {\ref{NumericalStudy}}. Conclusions are contained in Sect. \ref{Conc}.

\section{Preliminaries and Problem Formulation}\label{MoForm}
In this section, we introduce notations and preliminary results needed in our
derivations. Along the way, we describe and develop the problem formulation. 

\subsection{Graph Theory}
A weighted, directed graph can be defined by a triple $\mathcal{G} \triangleq (\mathcal{V},\mathcal{E},\mathcal{W})$. $\mathcal{V}\triangleq \{v_1,v_2,...,v_N \}$ represents a set of $N$ nodes. Define $\mathcal{N}\triangleq \{1,...,N\}$. A set of directed edges is denoted by $\mathcal{E}\subseteq\mathcal{V}\times \mathcal{V}$. The set of in-neighbors of node $i$ is defined as $\mathcal{N}_{i}^{in}\triangleq \{j|j\in \mathcal{V},(j,i)\in \mathcal{E}\}$. Denote by $|\cdot|$ the cardinality of a set. So, the in-degree of $v_i$ is $|\mathcal{N}_{i}^{in}|$. Similarly, the set of out-neighbors of $v_i$ is $\mathcal{N}_{i}^{out}\triangleq\{j|j\in \mathcal{V},(i,j)\in\mathcal{E}\}$. {The out-degree of $v_i$ is $|\mathcal{N}_{i}^{out}|$.} The weight adjacency matrix $\mathcal{G}$ is denoted by an $N\times N$ matrix $\mathcal{W}\triangleq[w_{ij}]$ where $w_{ij}$ refers to the weight of the edge from node $i$ to $j$. We assume that graph $\mathcal{G}$ has no self-loops.

We denote the original weight adjacency matrix by $\mathcal{W}^{o}=[w_{ij}^o] \in \mathbb{R}^{N \times N}$. Let $\mathcal{N}_{i,o}^{out}$($\mathcal{N}_{i,o}^{in}$) be the set of out-neighbors (in-neighbors) under the original optimal weight pattern $\mathcal{W}^o$. 

\subsection{Virus Spreading Model}

With the fact that cyber network nodes do not have human-like autoantibody/vaccination which can prevent individual from being infected again, we study the so-called susceptible-infected-susceptible (SIS) models. Consider a population of $N$ agents. Each agent can be either susceptible (S) or infected (I). Infected individuals infect others at rate $\beta_i\geq 0$. The intensity of interaction between $v_i$ and $v_j$ is described by the weight $w_{ij}\in \mathbb{R}$. Denote $\mathbf{w}_{i}=(w_{i1},...,w_{iN})'\in\mathbb{R}^N$. We assume that the weight is bounded by $\bar{w}_{ij} \in \mathbb{R}$. If $v_i \in \mathcal{V}$ is susceptible
while $v_j \in\mathcal{V}$ is infected, there is possibility that $v_i$ will be infected after the interaction. In addition,  each infected agent returns to the susceptible state at some rate $\sigma_i$. The state of a node $i$ at time $t\geq 0$ is a binary random variable $X_i(t) \in \{0,1\}$, with $X_i(t)=0$ ($X_i(t)=1$), indicating that agent $i$ is susceptible (infected). The state vector of all $N$ agents is denoted by $X(t)=(X_1(t),X_2(t),...,X_N(t))' \in \{0,1\}^N$. With the adaptive weight $w_{ij}(t)$ from agent $i$ to $j$, the stochastic state transitions of node $v_i$ from time $t$ to $t+\Delta t$ can be written as follows:
\begin{equation}\label{MarkovChain}
\begin{aligned}
&\mathbb{P}(X_i(t+\Delta t)
=1|X_i(t)=0,X(t))\\
&=\sum\limits_{j=1}^N w_{ij}\beta_j X_j(t)\Delta t + o(\Delta t),\\
&\mathbb{P}(X_i(t+\Delta t)=0|X_i(t)=1,X(t))=\sigma_i \Delta t + o(\Delta t).
\end{aligned}
\end{equation}

{The model (\ref{MarkovChain}) is computationally challenging under large-scale networks due to the exponentially increasing state space.} Hence, we resort to mean-field approximation of the Markov process \cite{Pare2017, Draief2010, VMP2009}. Denote $x_i(t) \in [0,1]$ as the probability of agent $i$ being infected at time $t$. The mean-field approximation then provides

\begin{equation}\label{EpDy}
    \dot{x}_i(t)=(1-x_i(t))\sum\limits_{j=1}^N w_{ij}(t)\beta_j x_j(t)-\sigma_i x_i(t),
\end{equation}
for $i=1,2,...,N$.
To write this dynamics equation in a more compact form, denote $\mathbf{x}(t)=(x_1(t),...,x_N(t))'$. We have
\begin{equation}\label{dynamic2}
\dot{\mathbf{x}}(t) =G(\mathbf{x}(t),\mathbf{W}(t)),
\end{equation}
where $G(\cdot,\cdot):\mathbb{R}^N \times \mathbb{R}^{N \times N} \rightarrow \mathbb{R}^N$ which can be written as
$G(\mathbf{x}(t),\mathbf{W}(t))=(W(t)B - D)\mathbf{x}(t) - X_d(t)W(t)B\mathbf{x}(t)$
where $W(t) = {[{w_{ij}(t)}]_{N \times N}}$, $B=diag(\beta_1,...,\beta_N)$, $D=diag(\sigma_1,...,\sigma_2)$, and $X_d(t)=diag(x_1(t),...,x_N(t))$.

{According to the discussion in \cite{Pare2017}, the $n$-intertwined model (\ref{EpDy}) gives an upper-bound for the exact probability of infection, $x_i(t)$. However, the mean-field approximation consider herein, while it is an approximation, is well constructed because the scale of networks, i.e., N in our model is large and we focus on the cases where $\beta/\sigma$ is above the threshold \cite{VMP2009}. }

The graph and the epidemic spreading process can be viewed as physical constraints. The agents in the network are coupled by these constraints while trying to minimize their own cost. {Such behaviors lead to differential games over networks, which will be introduced in the following section.}

\subsection{Differential Game Over Networks}

As we mentioned in Section \ref{Introduction}, the self-interested agents aim to minimize their own costs. One cost arise from malfunction caused by infection. Another cost for agent $i$ is to describe inefficiency or degradation of system performance caused by deviation from the original weight $w_{ij}^o$ for all $j\in\mathcal{N}$. We consider the original weight as an optimal weight under which the agent can achieve the most benefit.

For agent $i$, the infection cost function, given by $f_i:[0,1]\rightarrow \mathbb{R}^+$, is a function of $x_i(t) \in [0,1]$. $f_i$ is assumed to be monotonically increasing to capture the loss of being infected. A weight cost function for edge from $i$ to $j$ is given by $g_{ij}(w_{ij}(t)-w_{ij}^o)$ where $g_{ij}:\mathbb{R}\rightarrow \mathbb{R}^+$ is convex. The function satisfies $g_{ij}(w)=0$ at and only at $w=0$ for all $i,j\in\mathcal{N}$ because {the original weight is optimal to the agent when there is no infection. It is optimal in terms of the tradeoff between price and performance.} The marginal cost of deviation from the optimal weight will increase as the distance from the adapted weight to the optimal weight increases. Considering a time duration from $0$ to $T \geq 0$, the cost function of agent $i$ during time interval $[0,T]$ is given as follows by
\begin{equation}\label{costfunction}
J_i=\int\limits_{0}^T f_i(x_i(t))+\sum\limits_{j=1}^N g_{ij}(w_{ij}(t)-w_{ij}^o) dt . 
\end{equation}

As each node determines its own weight adaptation policy, it naturally leads to a differential game framework defined as follows.  Consider $N$ agents in the network as $N$ players with an index set $\mathcal{N}=\{1,...,N\}$. The duration of the evolution of the game is given by the time interval $[0,T]$. Denote $\mathbf{x}(t)=(x_1(t),...,x_N(t))'$. Let $\mathcal{X}=\{x\in\mathbb{R}^N| x_i\in [0,1],\forall i \in \mathcal{N}\}$ be the permissible set of the states. For each fixed $t\in [0,T]$, $\mathbf{x}(t)\in \mathcal{X}$. Let $\mathbf{w}_i(t)=(w_{i1},...,w_{iN})$ be the controls of player $i$. The admissible control set for player $i$ is $S_i=\{0 \leq w_{ij} \leq \bar{w}_{ij},\forall j\in \mathcal{N}\}$, i.e., for each fixed $t\in [0,T]$, $\mathbf{w}_i \in S_i \subset \mathbb{R}^N$. A differential equation is given by (\ref{EpDy}) whose solution describes the state trajectory of the game corresponding to the $N$-tuple of control functions $\{\mathbf{w}_i(t),0\leq t\leq T,i\in N\}$ and the given initial state $\mathbf{x}_0\triangleq(x_1(0),...,x_N(0))'=(x_{10},...,x_{N0})'$. Define a set-valued function $\eta_i(\cdot)$ for each $i\in N$ to characterize the information pattern of player $i$. We consider the open-loop pattern in our case where $\eta_i(t)=\{\mathbf{x}_0\},t\in[0,T]$. We can state our problem as the following differential game problem:
\begin{equation}\label{DifGam}
\begin{aligned}
\min\limits_{\mathbf{w}_i \in S_i}\ \ &J_i(\mathbf{w}_i)  \\
\textrm{s.t.~}\ &\textrm{(\ref{EpDy})},\\
\end{aligned}
\end{equation}
where $J_i(\cdot):\mathbb{R}^N\rightarrow \mathbb{R}$ and $x_i(0)=x_{i0}, \ i=1,2,...,N$. Each player aims to find a control policy $\mu_i(t, x_{i0})$ to generate a weight trajectory $w_i(t)$. Such control policies are open-loop ones that depend on the initial condition of the individual state. 
\begin{remark}
The game defined by (\ref{DifGam}) is a differential game over networks where the cost only depends on their own state and controls. Nodes interact with their neighbors. The network topology is captured by $w_{ij},i\in\mathcal{N},j\in\mathcal{N}$. The time-varying property of the network is described by $w_{ij}(t)$ for $t\in[0,T]$.
\end{remark}
{
\begin{remark}
Information structure determines the state information gained and recalled by players at time $t$. {The reasons why we adopt open-loop policies} are three-fold. First, the obtained open-loop policy can be implemented as a feedback policy \cite{TB1999} as is shown in Section \ref{AnaRes}. Since the dynamics (\ref{dynamic2}) is determined, the state at any time can be computed and used to determine the control policy. Second, to obtain a strongly time-consistent optimal and individual feedback policies, we have to resort to techniques of dynamic programming. However, a direct application of dynamic programming will not yield an individual feedback policy. Also, computation of the feedback control law derived from Hamilton\textendash Jacobi\textendash Bellman equation requires solving nonlinear PDEs which increases the difficulty of distributed implementation. Third, obtaining open-loop policy resorts to maximum principle which well presents the structure of the optimal solution. This helps us to analyze the inefficiency of the NE and obtain a penalty function to achieve social optimum as is shown in Section \ref{EffCom}.
\end{remark}
}

\section{Analytic Results}\label{AnaRes}
The solutions to the $N$-person non-cooperative nonzero-sum differential game (\ref{DifGam}) played with an open-loop information structure are open-loop Nash equilibria.
\begin{definition}
The weight adaptation trajectories or say the control trajectories $\{\mathbf{w}^*_i$, $i\in \mathcal{N}\}$ constitute an open-loop NE solution of the differential game (\ref{DifGam}) if the inequalities
\begin{equation}\label{DefNE}
\begin{aligned}
J_1(\mathbf{w}^*_1,\mathbf{w}^*_2,...,\mathbf{w}^*_N)&\leq J_1(\mathbf{w}_1,\mathbf{w}^*_2,...,\mathbf{w}^*_N) \\
J_2(\mathbf{w}^*_1,\mathbf{w}^*_2,...,\mathbf{w}^*_N)&\leq J_2(\mathbf{w}^*_1,\mathbf{w}_2,...,\mathbf{w}^*_N) \\
&\vdots \\
J_N(\mathbf{w}^*_1,\mathbf{w}^*_2,...,\mathbf{w}^*_N)&\leq J_N(\mathbf{w}^*_1,\mathbf{w}^*_2,...,\mathbf{w}_N) \\
\end{aligned}
\end{equation}
hold for all control trajectories $\mathbf{w}_i(t)\in S_i,t\in [0,T]$. We denote $x_i^*(t),t\in[0,T]$ the associated state trajectory for $i\in\mathcal{N}$.
\end{definition}
The definition states that at open-loop NE, no agents have incentive to deviate unilaterally away from the optimal trajectory from time $0$ to time $T$.

To obtain the necessary conditions for the open-loop NE, we make two mild assumptions.
\begin{assumption}\label{AssF}
For each $i\in\mathcal{N}$, the infection cost function $f_i(\cdot)$ is to be of $C^1$ class.
\end{assumption}

\begin{assumption}\label{AssG}
For each $i,j\in\mathcal{N}$, the weight deviation cost function $g_{ij}(\cdot)$ is to be of $C^1$ class.
\end{assumption}

Each player $i\in\mathcal{N}$ can decide to receive data or packets from any other agent. The following observation narrows down the set of possible solutions of the open-loop NE.

\begin{observation}\label{CutDownSpace}
If $\{u^*_{ij}(t),i\in\mathcal{N},j\in \mathcal{N}_{i,o}^{out}\}$ is an open-loop NE solution for the following differential game
\begin{equation}\label{DifGam2}
\begin{aligned}
&\min\limits_{\mathbf{u}_i \in U_i} J_i=\int\limits_{0}^T f_i(x_i(t))+\sum\limits_{j\in\mathcal{N}_{i,o}^{out}} g_{ij}(u_{ij}(t)-w_{ij}^o) dt  \\
& \textrm{s.t.~}\ \dot{x}_i(t)=(1-x_i(t))\sum\limits_{j\in \mathcal{N}_{i,o}^{out}} u_{ij}(t)\beta_j x_j(t)-\sigma_i x_i(t),\\ 
& \ \ \ \ \ x_i(0)=x_{i0},i=1,2,...,N,
\end{aligned}
\end{equation}
with $u_{ij}(t)\in[0,w_{ij}^o]$ for $i\in\mathcal{N},j\in \mathcal{N}_{i,o}^{out},t\in[0,T]$, and $\{w_{ij}^*(t), i\in\mathcal{N}, j\in \mathcal{N}\}$ is an open-loop NE solution for the differential game defined by (\ref{DifGam}), then we have 
\begin{equation}
    w^*_{ij}(t)=\begin{cases}
               u^*_{ij}(t)\ \ \ &\textrm{if}\ j\in\mathcal{N}_{i,o}^{out}\\
               0\ \ \ & \textrm{otherwise}
            \end{cases}
\end{equation}
for all player $i$ and for each $t\in[0,T]$.
\end{observation}

Proof: See Appendix $\ref{ProofCutDownSpace}$.

Observation \ref{CutDownSpace} simplifies the searching process for the open-loop NE. Instead of analyzing problem $(\ref{DifGam})$, we can focus on problem (\ref{DifGam2}) which contains a smaller admissible control set. Define $\mathbf{u}_i=\{u_{ij},j\in\mathcal{N}_{i,o}^{out}\}$. To be specific, the admissible control set of game problem (\ref{DifGam2}) for player $i$ is $U_i=\{0\leq u_{ij}(t) \leq w_{ij}^o,j\in \mathcal{N}_{i,o}^{out},t\in[0,T]\}$. From Theorem 5.1 of \cite{TB1999} and Lemma \ref{Lip}, the differential equation in (\ref{DifGam2}) admits a unique solution if the weight adaptation control is continuous in $t$.

Next, we discuss the derivation of candidate NE solutions for differential game (\ref{DifGam2}) when the information structure of the game is open-loop pattern. Utilizing techniques in optimal control theory, we arrive at the following result.

\begin{theorem}\label{NecCon}
For the $N$-person differential game (\ref{DifGam2}), we have assumptions \ref{AssF} and \ref{AssG}. Then, if $\{ \mathbf{u}^*_i(t),i\in \mathcal{N}\}$ is an open-loop NE solution, and $\{\mathbf{x}^*(t),0\leq t\leq T\}$ is the corresponding state trajectory, there exist $N$ costate functions $\mathbf{p}_i(\cdot):[0,T]\rightarrow \mathbb{R}^N,i\in\mathcal{N}$, whose $j$-th component is denoted by ${p}_{ij}(\cdot)$, such that the following relations are satisfied:
\begin{equation}\label{NecConDy}
\begin{aligned}
    \dot{x}_i^*(t)&=(1-x_i^*)\sum\limits_{j\in\mathcal{N}_{i,o}^{out}}u_{ij}^*(t)\beta_j x_j^*-\sigma_i x^*_i(t),\\
    & x_i^*(0)=x_{i0},\ \ \ \forall i\in\mathcal{N},
\end{aligned}
\end{equation}

\begin{equation}\label{ConHam}
\begin{aligned}
&\mathbf{u}_i^*(t)=\arg \min\limits_{\mathbf{u}_i \in U_i}\\ &H_i(t,\mathbf{p}_i(t),\mathbf{x}^*,\mathbf{u}_1^*(t),...,\mathbf{u}^*_{i-1},\mathbf{u}_i,\mathbf{u}^*_{i+1}(t),...,\mathbf{u}_N^*(t)),
\end{aligned}
\end{equation}

\begin{equation}\label{CosDy}
    \dot{\mathbf{p}}_i(t)=\Gamma_i(t,\mathbf{x}^*,\mathbf{u}^*_1,...,\mathbf{u}^*_N) \mathbf{p}_i(t)+\gamma_i(t),\ \ \ \mathbf{p}_i(T)=0,
\end{equation}
where 
\begin{equation}\label{HamNE}
\begin{aligned}
    &H_i(t,\mathbf{p}_i,\mathbf{x},\mathbf{u}_1,...,\mathbf{u}_N)
    \triangleq f_i(x_i(t))+ \sum\limits_{j\in \mathcal{N}_{i,o}^{out}} g_{ij}(u_{ij}-w_{ij}^o)\\
    &+\sum\limits_{j=1}^N p_{ij} \left\{(1-x_j)\sum\limits_{k\in \mathcal{N}_{j,o}^{out}}u_{jk}\beta_kx_k(t) + \sigma_j x_j(t)\right\},
\end{aligned}
\end{equation}
and $\Gamma_i$ is a matrix given by

\begin{equation}\label{CoMat}
\Gamma_{i,mn}=\begin{cases}
               \sum\limits_{j\in\mathcal{N}_{m,o}^{out}}u^*_{mj}(t)\beta_jx^*_j (t)+\sigma_m \ \ \ &\text{if}\ n=m\\
               -(1-x_n^*(t))u_{nm}^*(t)\beta_m\ \ \ &\textrm{if}\ n\in\mathcal{N}_{m,o}^{in}\\
               0 &\textrm{otherwise},
            \end{cases}
\end{equation}
$\gamma_i$ is a vector whose $i$-th component is $-df_i/dx_i$ and other components are zero, for $i\in \mathcal{N}$.
\end{theorem}
Proof: See Appendix \ref{ProofNecCon}.

{Note that $\Gamma_i$ turns out to be the same for different $i$. In later discussion, we shall omit the idex $i$.} Now, the dynamics of the costate function can be given as $\dot{\mathbf{p}}_i(t)=\Gamma(t)\mathbf{p}_i(t)+\gamma_i(t)$ for $i\in\mathcal{N}$ which sheds some light on the design for achieving social welfare in the following section. $\Gamma(t)$ is a $L$-matrix \cite{Young2014} for every $t\in[0,T]$ where the diagonal entries of $\Gamma(t)$ are positive and all off-diagonal entries are non-positive. {Therefore $\Gamma(t)$ is structurally in line with the graph Laplacian whose diagonal entries are the out-degrees of the $N$ agents \cite{Lin1974}.} That is for every zero or negative entry of the matrix $\Gamma(t)$, the corresponding entry of the graph Laplacian is zero or negative respectively and vice versa. If the original graph is a directed acyclic graph, $\Gamma(t)$ is a lower triangular matrix given the index of a proper permutation. Other than the topology information, $\Gamma(t)$ also contains the infection information. Note that even though we write the dynamics of the costates in an affine form, it is actually not affine which is because $\Gamma(t)$ depends on $\mathbf{x}^*(t)$ and $u^*_{ij},i\in\mathcal{N},j\in\mathcal{N}$ as we can see from (\ref{CoMat}) and $u^*_{ij}$ is dependent on $p_{ii}$ as we will show next in Theorem \ref{ConTheo}.

\begin{theorem}\label{ConTheo}
Define $\phi_{ij}(t):=p_{ii}(t)(1-x_i^*(t))\beta_j x_j^*(t)$ where $p_{ii}(\cdot)$ is the $i$th component of the costate function $\mathbf{p}_i(\cdot)$. The basic structure of the NE-based optimal weight control, i.e., the solution to (\ref{ConHam}), can be written as:
\begin{equation}\label{ConRule}
\begin{aligned}
&u^*_{ij}(t)=\\
&\begin{cases}
               0, &-\phi_{ij}(t)\leq g'_{ij}(-w_{ij}^o),\\
               (g'_{ij})^{-1}(-\phi_{ij}(t)),&g'_{ij}(-w_{ij}^o)<-\phi_{ij}(t)< g'_{ij}(0),\\
               w_{ij}^o, &-\phi_{ij}(t)\geq g'_{ij}(0),
            \end{cases}
\end{aligned}
\end{equation}
for $i\in\mathcal{N},j\in\mathcal{N}_{i,o}^{out}$.
\end{theorem}
Proof: See Appendix \ref{ProofConTheo}.

Condition (\ref{ConHam}) in Theorem \ref{NecCon} can thus be replaced by $(\ref{ConRule})$. Theorem \ref{NecCon} together with (\ref{ConRule}) provides a weight adaptation scheme where each agent adapts its weight to minimize the possibility of being infected and the loss of efficiency/interest. The weight of the edge from player $i$ to player $j$, controlled by player $i$, is based on the costate component $p_{ii}(t)$, player $i$'s own infection $x_i(t)$ and its out-neighbors infection. Apparently, the higher the infection level of agent $j$ is, the lower the weight of edge $(i,j)$ should be. As is shown in (\ref{CosDy}) and (\ref{CoMat}), $p_{ii}(t)$ is highly coupled and it contains information about the effect of the whole network.

\begin{remark}
Based on the structure of the optimal control (\ref{ConRule}), the dynamics of costates (\ref{CosDy}) as well as Lemma  \ref{Lip}, we can infer that the NE-based optimal control trajectory $u^*_{ij}(t)$ is continuous for every $i\in\mathcal{N}, j\in\mathcal{N}$ which means there is no switching in the optimal weight adaptation. 
\end{remark}

{\begin{remark}
From (\ref{ConRule}), we know the weight between agents $i$ and $j$ may be adapted to zero at certain time as one can see that $u^*_{ij}(t)=0$ if $-\phi_{ij}(t)\leq g'_{ij}(-\omega^o_{ij})$. That means the connection between agent $i$ and $j$ may be disconnected temporarily which will be restored according to Theorem \ref{CosTerm}. For agent $i$, if all its out-links have weight zero, i.e., $-\phi_{ij}(t)\leq g'_{ij}(-\omega^o_{ij})$ for all $j\in \mathcal{N}_{i,o}^{out}$, we can view this agent as being quarantined from infection. We say being quarantined from infection because there might still be in-links connecting to agent $i$ which means here, the concept of being quarantined is different from the concept in undirected graph. Besides, the weight adaptation scheme we proposed is different from quarantining in a sense that the weight adaptation scheme does not need to completely disconnect nodes from all other others but rather adjust weights to connect more loosely with particular nodes with a higher likelihood of infection.  
\end{remark}

}

\begin{corollary}
If $g_{ij}(\cdot)$ is concave, i.e., the marginal cost of deviation increases as the adapted weight becomes more far away from the optimal weight, the optimal control policy can be given as follows
\begin{equation}\label{ConRuleConcave}
\begin{aligned}
&u^*_{ij}(t)=\begin{cases}
               0, &\phi_{ij}(t) \geq \frac{g_{ij}(-w_{ij}^o)}{w_{ij}^o},\\
               w_{ij}^o,&\phi_{ij}(t) < \frac{g_{ij}(-w_{ij}^o)}{w_{ij}^o},
            \end{cases}
\end{aligned}
\end{equation}
for $i\in\mathcal{N},j\in\mathcal{N}_{i,o}^{out}$, where $\phi_{ij}(t)$ is defined in Theorem \ref{ConTheo}.
\end{corollary}
We can see that if $g_{ij}(\cdot)$ is concave, the optimal control policy switches between $0$ and $w_{ij}^o$. In this paper, we focus our study on the case when $g_{ij}(\cdot)$ is convex.

Before stepping into the numerical computation of the open-loop NE candidates, we go into further analysis and obtain other structural results that would be beneficial for more insightful understanding of the weight adaptation mechanism.

\begin{theorem}\label{CosTerm}
The costate function and the open-loop control trajectories have the following properties:
\begin{itemize}
\item[(i)] Along the open-loop NE trajectory, $p_{ij}(t)\geq 0$ holds for all $i,j \in \mathcal{N}, j\neq i$ and all $t\in[0,T]$. Furthermore, $p_{ii}(t)$ stays positive for all $i\in \mathcal{N}$ and all $t\in[0,T)$.

\item[(ii)] The open-loop NE control trajectory $u_{ij}^*(t)$, for $i\in\mathcal{N},j\in\mathcal{N}_{i,o}^{out}$, satisfies $u_{ij}^*(t)=w_{ij}^o$ at and only at $t=T$ and for $t<T$, $u_{ij}^*(t)<w_{ij}^o$.

\item[(iii)] If $|\mathcal{N}_{i,o}^{in}|=0$, i.e., the in-degree of player $i$ is zero in the original graph, under linear infection cost function $f_i(x_i(t))=\alpha_i x_i(t)$, the component $p_{ii}(t)$ is bounded above by $\alpha_i/\sigma_i$. That is, $p_{ii}(t) \leq \alpha_i/\sigma_i$ for $ t\in[0,T]$.

\item[(iv)] If $|\mathcal{N}_{i,o}^{out}|=0$, i.e., the out-degree of player $i$ is zero in the original graph, under linear infection cost function where $f_i(x_i(t))=\alpha_i x_i(t)$, the costate component $p_{ii}(t)$ is strictly monotonically decreasing over $t$.
\end{itemize}
\end{theorem}
Proof: See Appendix \ref{ProofCosTerm}.

Theorem \ref{CosTerm} indicates that during the time interval $[0,T)$, the agents, with an incentive to lower their own costs, adapt their weight accordingly to impede the spreading of virus. After the prescribed alert duration $[0,T]$, a recovery of topology is always on the way to meet the minimum cost. Also, from theorem \ref{CosTerm} (iii), we know for agent $i$ who has no in-neighbors, its out-link $u_{ij}^*(t)$ will never be $0$ if $\alpha_i \beta_j/  \sigma_i \leq g_{ij}' (-w_{ij}^o)$. This can be readily shown by $\phi_{ij}=p_{ii}(1-x_i^*(t))\beta_j x_j^*(t) \leq (\alpha_i/ \sigma_i)(1-x_i^*)\beta_j x_j^*(t) < \alpha_i \beta_j /\sigma_i \leq g_{ij}'(-w_{ij}^o)$.

\section{Inefficiency of Nash Equilibrium}\label{EffCom}

It is well known that the non-cooperative NE in nonzero-sum games is generally inefficient \cite{Dubey1986}. There is need to develop a mechanism to attain a higher social welfare or lower aggregate costs through cooperation behavior \cite{TB2011}. The notion of the price of anarchy has been introduced in \cite{Roughgarden2004} to quantify the inefficiency. In the network, the social cost is the aggregate costs of all players. {Let $\mathbf{u}=\{\mathbf{u}_1,...,\mathbf{u}_N\}$ where $\mathbf{u}_i(t) \in \mathbb{R}^{|\mathcal{N}_{i,o}^{out}|}$ be the weight control variable for the whole network with admissible set $U_o=\{\mathbf{u}: u_{ij}(t)\in[0,\omega_{ij}^o],\forall i\in\mathcal{N},j\in\mathcal{N}_{i,o}^{out},t\in[0,T] \}$.} Denote by $\mathbf{u}^o=\{\mathbf{u}_1^o,...,\mathbf{u}_N^o\}$ the social optimal solution. The social optimum can be attained by solving the optimal control problem: 
\begin{equation}\label{OpCon}
\begin{aligned}
    &\min\limits_{\mathbf{u}\in U_o} J_o=\int\limits_0^T \sum\limits_{i=1}^N f_i(x_i(t)){\color{blue}+}\sum\limits_{i=1}^N \sum\limits_{j\in \mathcal{N}_{i,o}^{out}}g_{ij}(u_{ij}(t)-w^o_{ij})dt\\
    &s.t.\ \dot{x}_i(t)=(1-x_i(t))\sum\limits_{j\in \mathcal{N}_{i,o}^{out}} u_{ij}(t)\beta_j x_j(t)-\sigma_i x_i(t),\\ &x_i(0)=x_{i0},i=1,2,...,N,
\end{aligned}
\end{equation}
where $J_o:\mathbb{R}^{|\mathcal{N}_{1,o}^{out}|}\times \cdots \times \mathbb{R}^{|\mathcal{N}_{N,o}^{out}|} \rightarrow \mathbb{R}$.
An application of maximum principle gives the following: the optimal control $\mathbf{u}^o(t)$ and corresponding trajectory $\mathbf{x}^o(t)$ must satisfy the following so-called canonical equations:
\begin{align}
\label{SoOp1}
\dot{x}_i^o(t)&=(1-x_i^o)\sum\limits_{j\in\mathcal{N}_{i,o}^{out}}u_{ij}^o(t)\beta_j x_j^o-\sigma_i x^o_i(t), {\color{blue}x_i^o(0)=x_{i0}},\\ 
\label{SoOp2}
    \dot\lambda(t)&=\Gamma(t,\mathbf{x}^o,\mathbf{u}^o_1,...,\mathbf{u}^o_N)\lambda(t)+\gamma,\lambda(T)=0,\\ 
\label{SoOp3}
    \mathbf{u}^o(t)&=arg\min\limits_{\mathbf{u}\in U_o} H(t,\mathbf{x}^o(t),\lambda(t),\mathbf{u}(t)),
\end{align}
for all $i\in\mathcal{N}$, where $\Gamma(t)$ is the same with the one given in (\ref{CoMat}) for the dynamics of the costate in the differential game problem and $\gamma(t)=[-f'_1(x_1(t)),-f'_2(x_2(t)),...,-f'_N(x_N)]'$, the Hamiltonian of the optimal control problem is defined as 
\begin{equation}\label{HamOp}
\begin{aligned}
&H(t,\mathbf{x}(t),\lambda(t),\mathbf{u}(t))\\
=&\sum_{i=1}^N f_i(x_i(t)){\color{blue}+}\sum_{i=1}^N \sum_{j\in \mathcal{N}_{i,o}^{out}}g_{ij}(u_{ij}(t)-w^o_{ij}) \\
+& \sum_{i=1}^N \lambda_i(t) \left\{(1-x_i(t))\sum\limits_{j\in \mathcal{N}_{i,o}^{out}} u_{ij}(t)\beta_j x_j(t)-\sigma_i x_i(t)\right\},
\end{aligned}
\end{equation}
and $\lambda(\cdot):[0,T]\rightarrow \mathbb{R}^N$ is the costate function, $\lambda_i$ is its $i$th component. {The Hamiltonian of the optimal control problem (\ref{HamOp}) is different from the individual Hamiltonian defined in (\ref{HamNE}). But they are related. The Hamiltonian of the optimal control problem includes the cost of all agents over the network instead of just individual's cost. Also, the costate $\lambda(t)$ corresponds to the state of all agents $\mathbf{x}(t)$. The counterpart of $\lambda_j(t), j\in\mathcal{N}$ in the individual Hamiltonian defined for the game problem is $p_{ij}(t),j\in\mathcal{N}$. Due to the similar structure of the Hamiltonian of the optimal control problem and the individual Hamiltonians for the game problem, after applying maximum principle, we obtain (\ref{SoOp1}-\ref{SoOp3}) that are in the same structure with (\ref{NecConDy}-\ref{CosDy}).
}

An optimal point can in principle be computed centrally by network operator to achieve social optimum. However, this will require the network operator to {be omniscient} and also not all the agents have incentives to adapt their connection weights based on the rule designed to minimize the aggregate costs. Also, for large-scale {network/system}, centralized solution gives rise to computational problems and implementability issues. So, centralized optimal control solution is impractical. To achieve social optimum in a distributed way, a mechanism needs to be designed on behalf of the network operator. The strategy for the network operator is to set penalties $c_i:=(c_i(t),t\in[0,T])$ so that the cost for player $i$ at time $t$ is $\hat{l}_i(t,\mathbf{x},\mathbf{u}) \triangleq f_i(x_i(t))+\sum_{j\in \mathcal{N}_{i,o}^{out}}g_{ij}(u_{ij}(t)-w^o_{ij})+c_i(t)$. To set a proper penalty for each player, we need to utilize the theory of potential differential games \cite{Morales2018} which is an extension of the potential game concept for static games \cite{Monderer1996}. The following is the definition of potential differential games.

\begin{definition}
A differential game with cost $\hat{J}_i=\int_{0}^T \hat{l}_i(t,\mathbf{x},\mathbf{u})dt$ and dynamics defined by (\ref{EpDy}) is an potential differential game if there exists a function $\pi:\mathcal{X}\times U_1\times \cdots \times U_N \times [0,T] \rightarrow \mathbb{R}$ that satisfies the following condition for every player $i \in\mathcal{N}$
\begin{equation}\label{PotenDef}
\begin{aligned}
    \int_{0}^T &\hat{l}_i(t,x_i,\mathbf{x}_{-i},\mathbf{u}_i,\mathbf{u}_{-i}) -\hat{l}_i(t,\hat{x}_i,\hat{\mathbf{x}}_{-i},\mathbf{v}_i,\mathbf{u}_{-i}) dt\\
    &=  \int_{0}^T \pi(t,x_i,\mathbf{x}_{-i},\mathbf{u}_i,\mathbf{u}_{-i}) -\pi(t,\hat{x}_i,\hat{\mathbf{x}}_{-i},\mathbf{v}_i,\mathbf{u}_{-i}) dt,
\end{aligned}
\end{equation}
for all $\mathbf{u}_i,\mathbf{v}_i \in U_i$, where $\mathbf{x},\hat{\mathbf{x}}\in\mathcal{X}$ are the corresponding states under controls $\{\mathbf{u}_i,\mathbf{u}_{-i}\}$ and $\{\mathbf{v}_i,\mathbf{u}_{-i}\}$ respectively. {Here, $\mathbf{u}_{-i}$ denote the collection of controls of all players except player $i$, i.e., $\mathbf{u}_{-i}=\{\mathbf{u}_1,...,\mathbf{u}_{i-1},\mathbf{u}_{i+1},...,\mathbf{u}_N\}$.
}
\end{definition}

If we can find $c_i(t)$ for every $i\in\mathcal{N}$ such that relation (\ref{PotenDef}) holds for $\pi= \sum_{i=1}^N f_i(x_i(t))-\sum_{i=1}^N \sum_{j\in \mathcal{N}_{i,o}^{out}}g_{ij}(u_{ij}(t)-w^o_{ij})$, then the differential game is a potential game whose open-loop NE solution is equivalent to the open-loop solution of optimal control problem defined by (\ref{OpCon}). The following theorem gives important insights about choosing proper penalties $c_i(t)$.

\begin{theorem}\label{SocNE}
Consider a differential game with penalties where the cost of player $i$ is given by 
\begin{equation}\label{Jhat}
\hat{J}_i=\int_{0}^T \hat{l}_i(t)dt=J_i+\int_{0}^Tc_i(t)dt,
\end{equation}
which is obtained by introducing a penalty term introduced to $J_i$ in (\ref{DifGam2}), and the constraint dynamics is in accordance with the differential game defined by (\ref{DifGam2}). Let $c_i(t)=\sum_{j\neq i}f_j(x_j(t))$. Then, {the differential game with (\ref{Jhat}) is a potential differential game} corresponding to the optimal control problem defined by (\ref{OpCon}), and if $\{ \mathbf{u}^*_i(t),i\in \mathcal{N}\}$ is an open-loop NE solution for new differential game with cost (\ref{Jhat}), and $\{\mathbf{x}^*(t),0\leq t\leq T\}$ is the corresponding state trajectory, the relations (\ref{SoOp1}) (\ref{SoOp2}) and (\ref{SoOp3}) hold for $\mathbf{u}^*$ and $\mathbf{x}^*$ with $\mathbf{u}^o$ replaced by $\{ \mathbf{u}^*_i(t),i\in \mathcal{N}\}$ and $\mathbf{x}^o$ replaced by $\mathbf{x}^*$.
\end{theorem}
Proof: See Appendix \ref{ProofSocNE}.

Theorem \ref{SocNE} indicates that with a proper choice of the penalties $c_i(t),i\in\mathcal{N}$, the necessary conditions for the open-loop NE solution of the differential game is aligned with the necessary conditions for the optimal solution of the social cost optimal control problem. The counterpart of condition (\ref{CosDy}) for the penalty-based differential game can be written as 
\begin{equation}\label{NewCosDy}
    \dot{\mathbf{p}}_i(t)=\Gamma(t)\mathbf{p}_i(t)+\hat{\gamma}_i(t)
\end{equation}
where $j$th component of $\hat{\gamma}_i(t)$ is $\partial(f_i + c_i)/\partial x_j$. In implementation, the system operator sends the penalty function to each agent. If the open-loop NE and the social optimal solution uniquely exist, the open-loop NE achieves the social optimum.

\begin{definition}
A directed graph is strongly connected if it contains {a directed path from $j$ to $i$ for every pair of vertices $\{(i,j):i\in\mathcal{N},j\in \mathcal{N},i\neq j\}$}.
\end{definition}

{In a directed graph, a directed path is a sequence of edges which join a sequence of vertices, but with the restriction that the edges all be directed in the same direction.} What we have developed so far in this section is for cases where the original weighted network is strongly connected. If the original network is not strongly connected, we have the following. 
\begin{definition}
Given a graph, if there {exists a directed path from vertex $j$ to vertex $i$}, we say $i$ is reachable from $j$. Denote by $\mathcal{R}_i \subseteq \mathcal{N}$ the set of vertices that $i$ can be reachable from.
\end{definition}

In graph theory, a single vertex is defined to connect to itself by a trivial path. We have $i\in\mathcal{R}_i$. If the graph is strongly connected, for every $i\in\mathcal{N}$, $\mathcal{R}_i = \mathcal{N}$. Otherwise, for some $i\in\mathcal{N}$, we have $\mathcal{R}_i \subset \mathcal{N}$. Denote $\mathcal{R}_{i,o}$ the counterpart of $\mathcal{R}_i$ for the original graph defined by $(\mathcal{V},\mathcal{E},\mathcal{W}^o)$.

\begin{corollary}\label{SocialSpe}
Consider the differential game with cost functions defined in (\ref{Jhat}) and dynamics given by (\ref{DifGam2}) as in Theorem \ref{SocNE}. Let $c_i(t)=\sum_{j\in\mathcal{R}_{i,o}\backslash \{i\}} f_j(x_j)$. Then, if $\{ \mathbf{u}^*_i(t),i\in \mathcal{N}\}$ is an open-loop NE solution for the new differential game, and $\{\mathbf{x}^*(t),0\leq t\leq T\}$ is the corresponding state trajectory, the relations (\ref{SoOp1}), (\ref{SoOp2}), and (\ref{SoOp3}) hold for $\mathbf{u}^*$ and $\mathbf{x}^*$ with $\mathbf{u}^o$ replaced by $\{ \mathbf{u}^*_i(t),i\in \mathcal{N}\}$ and $\mathbf{x}^o$ replaced by $\mathbf{x}^*$.
\end{corollary}
\begin{proof}
The proof of Corollary \ref{SocialSpe} simply follows from the proof for Theorem \ref{SocNE}.
\end{proof}

To illustrate Corollary \ref{SocialSpe}, we present an example in Appendix \ref{Corr2Examp}

\begin{figure}\centering
\vspace{-3mm}\includegraphics[width=8cm]{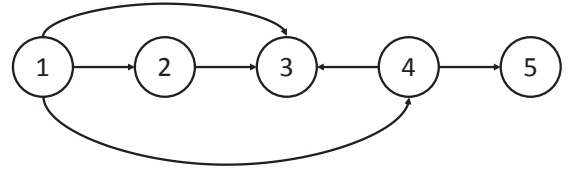}
\caption{A directed graph with $5$ vertices to illustrate the equilibrium solution of a $5$-person differential game over the network.}\label{DAG}
\end{figure}


\section{Algorithms and Case Studies}\label{NumericalStudy}
In this section, we provide the set-up information for the case studies. Besides, based on the equilibrium analysis and the optimal control analysis, an algorithm is proposed to compute the optimal weight adaptation trajectory for the system operator and the agents. 
\subsection{Preliminaries}
In the simulation, the infection cost function $f_i(\cdot)$ is given to be linear in $x_i(t)$, i.e., $f_i(x_i(t))=\alpha_ix_i(t)$. Here, we set $\alpha_i=\alpha,\forall i\in\mathcal{N}$. The weight adaptation cost is taken to be quadratic where $g_{ij}(u_{ij}-w_{ij}^o)=(1/2)d_{ij}(u_{ij}-w_{ij}^o)^2$ for all $i\in\mathcal{N},j\in\mathcal{N}_{i,o}^{out}$. Unless otherwise stated, let $\alpha_i=1, d_{ij}=d=0.2$ for all $i\in\mathcal{N},j\in\mathcal{N}_{i,o}^{out}$. Note that under this setting, assumptions $\ref{AssF}$ and $\ref{AssG}$ hold and $g_{ij}(\cdot)$ is even and convex.

{
\begin{table*}
\begin{center}
	\caption{Parameters used for numerical study unless otherwise stated.}\label{t1}
	\begin{tabular}{@{}ccccccccccc@{}} \toprule
		\multicolumn{11}{c}{\;\;\;\;\;\;\;\;\;\;\;\;\;Cost functions\;\;\;\;\;\;\;\;\;\;\;\;\;\;\;\;\;\;\;\;\;\;\;Network Topology \;\;\;\;\;\;\;\;\;\;\;\;\;\;\;\;\;\;\;\;\;\;\;\;\;\; Spreading} \\ 
		\cmidrule(r){1-4}
		\cmidrule(r){5-8}
		\cmidrule(r){9-11}
		$f_i(x_i)$ &$g_{ij}(u_{ij}-w^o_{ij})$ &$\alpha_i$ &$d_{ij}$ &$N$ &$w_{ij}^o$ &$\langle k^{out} \rangle$ &$\langle k^{in}\rangle$  &$\beta_i$ &$\sigma_i$ &$T$ \\
		\cmidrule(r){1-11}
		$\alpha_ix_i$ &$\frac{1}{2}d_{ij}(u_{ij}-w_{ij}^o)^2$ &$1$ &$0.2$ &$150$ &$1$
		& $7.545$& $7.545$ &$0.04$ &$0.1$ &20  \\
		\bottomrule
	\end{tabular}
\end{center}
\end{table*}}

The original network in the simulation is {a bi-directional scale-free network} with $150$ agents generated based on the Barab\'{a}si-Albert model \cite{Albert2002,Barabasi1999}. {We select this model since many kinds of computer networks}, including the internet and the web graph of the World Wide Web, have scale-free properties. We generate the network by following the growth and preferential attachment properties given in section VII of \cite{Albert2002}. For simplicity, the original weight is set to be $w^o_{ij}=1$ for all edge $(i,j)$. Let $\langle k^{in}\rangle$ $\langle k^{out} \rangle$ be the average in-degree (out-degree) of the network we generated. We have $\langle k^{in}\rangle=\langle k^{out}\rangle=7.545$. 

For simplicity, we take same infection rates and curing rates for all players. Unless otherwise stated, let $\beta_i=\beta=0.04$ and $\sigma_i=\sigma=0.1$. From the result in \cite{Ganesh2005} and the fact that the largest real part of the eigenvalues of matrix $(\mathcal{W}^o B-D)$ is $0.5368$, we say that the virus epidemic will outbreak in the original network. The initial infection level is also set to be the same for all players, $x_{i0}=0.16$ for all $i\in\mathcal{N}$. Table \ref{t1} is a summary of the setups.

\subsection{Computational Algorithm}

{\tiny
\begin{algorithm}[tb]
   \caption{DVR}
   \label{DVR}
\begin{algorithmic}\footnotesize
\STATE {\bfseries Input:} $x_{i0},i\in\mathcal{N}$; {$\beta_i,\sigma_i,i\in\mathcal{N}$; $f_i(\cdot),g_{ij}(\cdot),i\in\mathcal{N},j\in \mathcal{N}_{i,o}^{out}$; $\mathcal{W}^o$; $\epsilon$.} 

\STATE {\bfseries Step 1:} Each {$i\in\mathcal{N}$ selects $\mathbf{u}_i(t) \in U_i$ arbitrarily.}

\STATE {\bfseries Step 2:} Forward part: obtain $x_i(t)$ from $x_{i0}$ and (\ref{NecConDy}) for each $i\in\mathcal{N}$.

\IF{solving social problem (\ref{OpCon})}
\STATE {\bfseries Step 3:} Backward part: obtain $\lambda(t)$ from $ \dot{\mathbf{p}}_i(t)=\Gamma(t)\mathbf{p}_i(t)+\hat{\gamma}_i(t)$ and $\mathbf{p}(T)=0$.
\ELSE
\STATE {\bfseries Step 3$'$:} Backward part: Obtain $\mathbf{p}_i(t),i\in\mathcal{N}$ from $\dot{\mathbf{p}}_i(t)=\Gamma(t)\mathbf{p}_i(t)+\gamma_i(t)$ and $\mathbf{p}_i(T)=0$.
\ENDIF
\STATE {\bfseries Step 4:} Control update: Update control $\hat{\mathbf{u}}_i(t)$ based on (\ref{ConRule}).
\STATE {\bfseries Step 5:} Control policy check:
\IF {$\left\| {\hat{\mathbf{u}}-\mathbf{u}} \right\|_\infty \geq \epsilon$} 
    \STATE Set $\mathbf{u}= \hat{\mathbf{u}}$. Return to \textbf{Step 2}.
\ELSE
    \STATE Set $\mathbf{u}^*_i=\mathbf{\hat{u}}_i$. $\mathbf{u}^*_i$ is the optimal weight adaptation scheme for agent $i$.
\ENDIF
\end{algorithmic}
\end{algorithm}
}

Note that we aim to propose an implementable distributed virus resistance algorithm (DVR algorithm). Based on the algorithm proposed in \cite{David1972} for computation of open-loop NE for nonzero-sum differential games, we present, in algorithm \ref{DVR}, the DVR algorithm to compute the candidate open-loop NE solutions for the differential game described by (\ref{DifGam2}) and the penalty-based differential game defined in Theorem \ref{SocNE}. The solution of the penalty-based differential game is inline with the solution of the optimal control problem defined in (\ref{OpCon}).

{Initially, the input data includes initial infection data: $x_{i0}$ for all $i\in\mathcal{N}$; infection rate $\beta_i$, recovery rate $\sigma_i$  for all $i\in\mathcal{N}$; the original topology $\mathcal{W}^o$; the cost functions $f_i(\cdot),g_{ij}(\cdot)$ for all $i\in\mathcal{N},j\in \mathcal{N}_{i,o}^{out}$ and a stopping value $\epsilon$ to stop the algorithm.} In the first step, each player arbitrarily selects {a continuous control trajectory within the admissible control set} $U_i$ for every out-link it has: $u_{ij}(t)$ for each $i\in\mathcal{N}$, for all $j\in\mathcal{N}_{i,o}^{out}$, and reports the weight adaptation scheme $\mathbf{u}_i(t)$ to the network operator. In step $2$, each player utilizes the initial infection data $x_{i0},i\in\mathcal{N}$ and the control policy $\mathbf{u}_i,i\in\mathcal{N}$, solve (\ref{NecConDy}) forward in time to obtain $x_i, i\in\mathcal{N}$ and report it to the network operator. If the system aims to achieve the social optimal control problem, then the algorithm goes into step $3$. Otherwise, the algorithm steps into step $3'$. In step $3$, the system operator utilizes the reported $\mathbf{u}_i$ and $x_i,i\in\mathcal{N}$, the infection damage cost $f_i(\cdot),i\in\mathcal{N}$ to compute $\mathbf{p}(t)$ backward based on (\ref{NewCosDy}) and sends $\mathbf{p}_i(t)$ back to the corresponding player $i$. In step $3'$, the system operator utilizes the reported $\mathbf{u}_i$ and $x_i,i\in\mathcal{N}$, the infection damage cost $f_i(\cdot),i\in\mathcal{N}$, computes $\mathbf{p}_i(t),i\in\mathcal{N}$ backward based on (\ref{CosDy}) and sends $\mathbf{p}_i(t)$ back to the corresponding player $i$. In the next step, each player updates its control based on (\ref{ConRule}) which only requires its out-neighbors infection information and reports the updated control policy to the network operator. Denote by $\hat{\mathbf{u}}_i(t)$ the updated control policy. If $\Vert {\hat{\mathbf{u}}-\mathbf{u}} \Vert_\infty \geq \epsilon$, the algorithm moves back to step $2$. Otherwise, the latest updated policy $\mathbf{\hat{u}}_i$ is the optimal control policy for agent $i$.


\subsection{Numerical Results}

In this subsection, we present the numerical results. First, we show the dynamics of the costate function for all players. Then, we show the evolution of the weight adaptation, the infection and the costate of selected agents to see individuals' behaviors. Second, we give the comparisons between the optimal control based-weight adaptation scheme (this scheme is equivalent to the penalized differential game based-weight adaptation scheme) and the differential game-based weight adaptation scheme. The optimal control based adaptation scheme is from solving optimal control problem (\ref{OpCon}). The two schemes together with the case of weight adaptation. are compared in terms of the total cost and the infection level of the whole network.

From $(\ref{ConRule})$ and $\phi_{ij}(t):=p_{ii}(t)(1-x_i^*(t))\beta_j x_j^*(t)$, we know that the weight adaptation of player $i$ is based on its own infection, its out-neighbors, and the costate component $p_{ii}$. {The infection of player $i$ and its neighbors are just local information.} From (\ref{CosDy}) and (\ref{CoMat}), {we can see the effect of the whole network's situation is conveyed by costate component $p_{ii}$ to the weight adaptation strategy of player $i$.} Thus, we investigate the dynamics of $p_{ii}$ in Fig. \ref{CostateDynamics}, where the costate component $p_{ii}$'s dynamics for all agents are plotted. As we can see, $p_{ii}(t)$ is positive for all $i\in\mathcal{N}$ during the whole time interval which corroborates Theorem \ref{CosTerm}. For most of the players, the value of $p_{ii}$ is high at the very beginning and then decreases to $0$. One interpretation is that players are more sensitive at the beginning to their out-neighbors infection and tend to cut their weights more heavily.

To see individual behaviors and states, we rank the agents based on their out-degrees. Agent $1$ has the largest out-degree. From the first plot of Fig. \ref{IndWei}, agent $1$ is more likely to be infected due to its large degree. We can see that all weights equal $1$ at and only at $t=T=20$, which corroborates Theorem \ref{CosTerm}. The weight $u_{150,1}(t)$ is reduced to $0$ for some time. This phenomenon occurs because the costate and its out-neighbors' infection levels are high during that time period. The third plot shows that agents with higher out-degrees reduce less weight. Usually, one suppose to cut more weights on highly connected nodes to slow the infection propagation. However, the obtained weight adaptation scheme in this paper is a result of considering both the infection and the loss of efficiency of the network agents. There is a trade-off between maintaining the network's performance and lowering the infection. So, the agents with higher out-degrees may cut less weight to maintain the performance/function.
\begin{figure}\centering
\includegraphics[width=7.6cm]{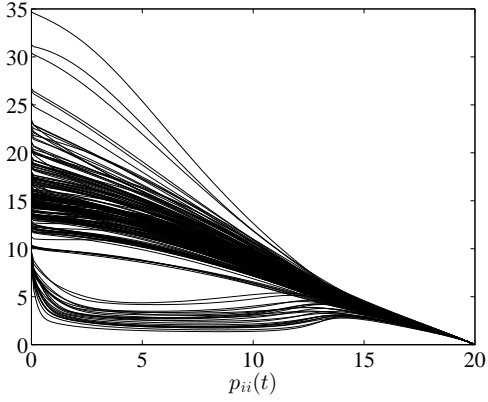}
\caption{The dynamics of costate $p_{ii}(t)$ for all players $i\in \mathcal{N}$.} \label{CostateDynamics}
\end{figure}
\begin{figure}\centering
\includegraphics[width=8.8cm]{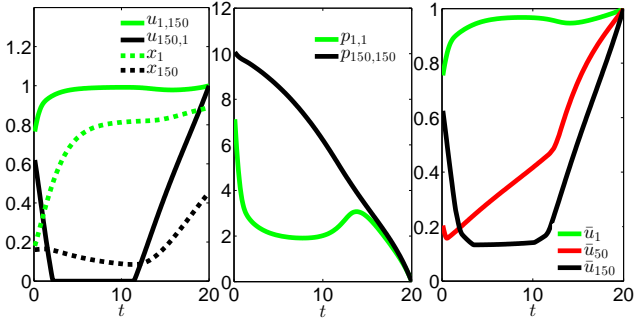}
\caption{The weight adaptation, the infection and the costate of agents $1$, $50$ and $150$ over time. The first plot shows the weight adaptation of edge $(1,150)$ and edge $(150,1)$ and the infection of agent $1$ and $150$. The second plot of Fig. \ref{IndWei} shows the costate dynamics of the two agents. The third gives the average weight adaptation of agents $1$, $50$ and $100$ where $\bar{u}_i=\sum_{j\in\mathcal{N}_{i,o}^{out}} u_{ij}/|\mathcal{N}_{i,o}^{out}|$.} \label{IndWei}
\end{figure}

{
\begin{figure}\centering
\includegraphics[width=8.0cm]{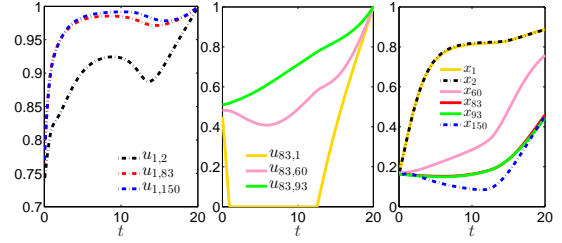}
\caption{The weight adaptation of links between fixed source agent and its different neighbors. The weight adaptation of links connecting from agent $1$ to its out-neighbors, agents $2$, $83$ and $150$ is plotted in the first figure. In the second figure, the weight adaptation of links connecting from agent $83$ to its out-neighbors, agents $1$, $60$ and $93$. The infection level of all agents involved is plotted in the last figure.} \label{SameNodeDifLink}
\end{figure}
}

\begin{figure}\centering
\includegraphics[width=8.7cm]{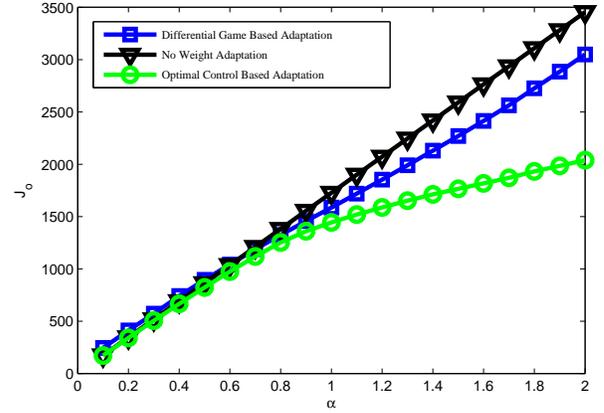}
\caption{The total cost $J_o$ versus different infection cost functions $k$ under the differential game-based weight adaptation scheme, optimal control based weight adaption scheme and no weight adaptation scheme.} \label{alphaToJo}
\end{figure}

\begin{figure}\centering
\includegraphics[width=0.5\textwidth]{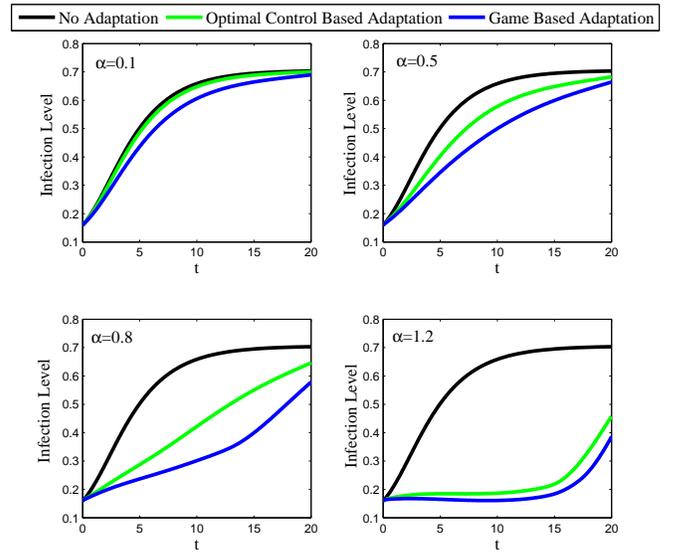}
\caption{The dynamics of the whole network's infection under differential game-based weight adaptation scheme, optimal control based weight adaption scheme, and no weight adaptation scheme.} \label{Spreading}
\end{figure}

{To show that each agent has heterogeneous weight adaptation to different neighbors, we present Fig. {\ref{SameNodeDifLink}}.
We can know from Fig. \ref{SameNodeDifLink} that agent $1$ adapts weights with his/her out-neighbors accordingly based on the evolution of the infection levels of his/her out-neighbors. As we can see, agent $1$ cuts more weight on neighbors with higher infection levels. For example, the infection level of agent $2$ is higher than agent $150$ all the time. Thus, weight $u_{1,2}$ is lower than $u_{1,150}$. Also, agent $83$ reduces its weight on agent $1$ to zero due to the latter's high infection level while its weight on agent $93$ remains above $0.5$.}

Here, we compare the NE-based weight adaptation scheme, the optimal control-based weight adaptation scheme, and no weight adaptation scheme. In Fig. \ref{alphaToJo}, we plot the total cost $J_o$ under the three schemes for different $\alpha$. We observe that no adaptation scheme cause the most total cost. For different values of $\alpha$, the NE-based scheme always incurs a higher cost than the optimal control-based scheme, which indicates the inefficiency of the NE solution. From the plot, we see that a higher $\alpha$ causes more inefficiency.

Fig. \ref{Spreading} is presented to show the virus-resistance of the proposed schemes. The black line shows the infection level for the case with no adaptation scheme, {the blue line shows the case with the game-based scheme, and the green line shows the case with the optimal control-based scheme.} Even though the game-based scheme is inefficient in terms of minimizing the total cost, it outperforms the optimal control-based scheme since the infection level under the game-based scheme is always lower than the infection level under the optimal control based scheme. {No matter in what case, the scheme we have proposed has proven to be virus-resistant and generated a lower total cost than the scheme without adaptation did.}

\section{Conclusion and Future Work}\label{Conc}
In this paper, we have established a differential game framework to develop decentralized virus-resistant mechanisms over complex networks. We have shown that weight adaptation policies allow nodes to change weights to mitigate their infection. {The differential game approach has captured the strategic and dynamic behaviors of a large number of self-interested agents} over time-varying networks. Each player adapts its weight based on its own infection and its out-neighbors infection. It has been observed that the higher levels of its out-neighbors' infection lead to lower weights. The effect of non-local behaviors on the adaptation strategy has been encoded in the costate function. We have discussed the inefficiency of the open-loop Nash equilibrium and {have proposed a penalty-based mechanism to achieve efficiency by imposing local costs induced by reachable nodes.} The differential game framework has enabled the design and implementation of a distributed algorithm over large-scale networks to control the macroscopic behaviors of the virus spreading over networks. Numerical examples have been used to illustrate the virus-resistance of the proposed scheme and the inefficiency of the Nash equilibrium. The differential game approach achieves a better performance than its centralized counterpart in terms of the mitigation of virus spreading. One future direction for this work would be to study the steady behavior of long-term virus-resistance scheme where the duration of virus spreading is sufficiently long.


%
\appendices

\section{Lemmas}

\begin{lemma}\label{Lip}
The dynamics equation $G(\mathbf{x}(t),\mathbf{W}(t))$ is uniformly Lipschitz in $\mathbf{x}$ and $\mathbf{w}_i$ for each $i\in\mathcal{N}$.
\end{lemma}
\begin{proof}
From (\ref{dynamic2}), we have
\begin{equation}\label{LipEq}
\begin{aligned}
&{\Vert {G(\mathbf{x}(t),W(t)) - G(\hat{\mathbf{x}}(t),W(t))} \Vert_\infty }\\
{\rm{ = }}&\Vert (W(t)B - D)\mathbf{x}(t) - (W(t)B - D)\hat{\mathbf{x}}(t)\\
&- X_d(t)W(t)B\mathbf{x}(t) + \hat{X}_d(t)W(t)B\hat{\mathbf{x}}(t) \Vert_\infty \\
\le &{\Vert {({I_N} - X_d(t))W(t)B(\mathbf{x}(t) - \hat{\mathbf{x}}(t))} \Vert_\infty } \\&+ {\Vert {D(\mathbf{x}(t) - \hat{\mathbf{x}}(t))} \Vert_\infty }+ {\Vert {(X_d(t) - \hat{X}_d(t))W(t)B\mathbf{x}(t)} \Vert_\infty }\\
\le &2{\beta _M}{\Vert {W(t)} \Vert_\infty }{\Vert {\mathbf{x}(t) - \hat{\mathbf{x}}(t))} \Vert_\infty } + {\delta _M}{\Vert {\mathbf{x}(t) - \hat{\mathbf{x}}(t))} \Vert_\infty }\\
 \le &\max \{ 2{\beta _M},{\delta _M}\} {\Vert {\mathbf{x}(t) - \hat{\mathbf{x}}(t))} \Vert_\infty }(1 + {\Vert {W(t)} \Vert_\infty }),
\end{aligned}
\end{equation}
{where ${\beta _M} \coloneqq {\max _i}{\beta _i}$ and ${\delta _M}$ $\coloneqq$ ${\max _i}{\delta _i}$, and} $I_N\coloneqq\textrm{diag}(1,1,\cdots,1)$. Since $w_{ij}$ is bounded by $\bar{w}_{ij}$, $(1 + {\left\| {W(t)} \right\|_\infty })$ is bounded. So, $G$ is uniformly Lipschitz in $\mathbf{x}$. The proof for $G$ is uniformly Lipschitz in $\mathbf{w}_i$ for all $i\in\mathcal{N}$  can be obtained by following the similar steps.
\end{proof}

\begin{lemma}\label{NotDis}
Let $x_i$, $i\in \mathcal{N}$ be the corresponding solution to the ODEs in (\ref{EpDy}). For all $i \in \mathcal{N}$, given $0<x_i(0)\leq 1$, $x_i(t)\in(0,1)$ holds for all $t\in(0,\infty)$.
\end{lemma}
\begin{proof}

The proof follows Lemma 1 of \cite{Pare2017}. It is clear that $x_i(t)$ is a continuous function of time. When $x_i(0)=1$, then from (\ref{EpDy}), we have $\dot{x}_i(0)<0$, which means once $x_i(t)$ reaches $1$, it cannot stay there. On the other hand, when $x_i(0)\in(0,1)$, the solution $x_i(t)$ would always lie in $(0,1)$. Otherwise, suppose that there exists $t_1$ such that $x_i(t_1)=0$ or $x_i(t_1)=1$. In the first case, note that $\dot{x}_i(t)\geq - \sigma_i x_i(t)$ holds for the time interval $(0,t_1]$, which gives $x_i(t_1)\geq x_i(0) e^{-\sigma_it_1} >0$. It yields a contradiction. In the second case where $x_i(t_1)=1$, we have $x_i(t)<1$ over time interval $[0,t_1)$. So, we have $\dot{x}_i(t_1^-)\geq 0$ which contradicts the fact obtained from (\ref{EpDy}) that $\dot{x}_i(t_1)<0$.
\end{proof}

\section{proof}

\subsection{Proof of Observation \ref{CutDownSpace}}\label{ProofCutDownSpace}
\begin{proof}
Given the original weight pattern $\mathcal{W}^o$, suppose that $w_{ij}^o=0$, i.e., $j\notin \mathcal{N}_{i,o}^{out}$ and $w^*_{ij}(t)>0$. Obviously, player $i$ can lower his own cost by deviating $w_{ij}(t)$ from $w^*_{ij}(t)$ to $0$, which contradicts the fact that $w^*_{ij}(t)$ is the open-loop NE. A similar statement can be made for the case when $w^*_{ij}>w_{ij}^o$.
\end{proof}

\subsection{Proof of Theorem \ref{NecCon}}\label{ProofNecCon}

\begin{proof}
Based on Theorem 6.11 in \cite{TB1999}, conditions (\ref{NecConDy}) and (\ref{ConHam}) are directly derived. To obtain (\ref{CosDy}), we have
\begin{equation}\label{GetCost}
\begin{aligned}
\dot{p}_{ii}(t)=&-\frac{\partial H_i}{\partial x_i}=-f'_i(x^*_i(t))\\
&+p_{ii}\left\{\sum\limits_{j\in\mathcal{N}_{i,o}^{out}}u^*_{ij}(t)\beta_jx^*_j(t)+\sigma_i \right\}\\
&-\sum\limits_{j\in \mathcal{N}_{i,o}^{in}}p_{ij}(t)(1-x^*_j(t))u^*_{ji}(t)\beta_i,\\
\dot{p}_{ij}(t)=&-\frac{\partial H_i}{\partial x_j}=p_{ij}\left\{\sum\limits_{k\in\mathcal{N}_{j,o}^{out}}u^*_{jk}(t)\beta_kx^*_k(t)+\sigma_j \right\}\\
&-\sum\limits_{k\in \mathcal{N}_{j,o}^{in}}p_{ik}(t)(1-x^*_k(t))u^*_{kj}\beta_j.
\end{aligned}
\end{equation}
Reformulating (\ref{GetCost}), we obtain condition (\ref{CosDy}). As the terminal cost is unspecified and the final state is free, we have the transversality condition $\mathbf{p}_i(T)=0$.
\end{proof}

\subsection{Proof of Theorem \ref{ConTheo}} \label{ProofConTheo}

\begin{proof}
{By Assumption \ref{AssG} and the fact that $g_{ij},i,j\in\mathcal{N}$ is convex, we know that the Hamiltonian $H_i$ is differentiable and convex on $u_{ij}$ for every $j$. Also, the admissible control set is convex. Thus, the solution of (\ref{ConHam}) can be obtained by letting $\partial H_i(u_{ij})/ \partial u_{ij}=0$, i.e.,
$$
\begin{aligned}
\frac{\partial H_i(u_{ij}(t))}{\partial u_{ij}(t)}&=g'_{ij}(u_ij(t)-\omega_{ij}^o) + p_{ii}(t)(1-x_i(t))\beta_j x_j(t)\\
&=g'_{ij}(u_ij(t)-\omega_{ij}^o) + \phi_{ii}(t)\\
&=0.\\
\end{aligned}
$$
Suppose ${\partial H_i(\tilde{u}_{ij}(t))}/{\partial u_{ij}(t)}=0$. If $\tilde{u}_{ij}(t) \in [0,\omega_{ij}^o]$ which happens if and only if $g'_{ij}(-w_{ij}^o)<-\phi_{ij}(t)< g'_{ij}(0)$, then according to (\ref{ConHam}), $u^*_{ij}(t)=\tilde{u}_{ij}(t)=(g'_{ij})^{-1}(-\phi_{ij}(t))$. Otherwise, if $\tilde{u}_{ij}(t)<0$, $u^*_{ij}(t)=0$ while if $\tilde{u}_{ij}(t)>\omega_{ij}^o$, $u^*_{ij}(t)=\omega_{ij}^o$. Thus, we have the optimal control rule in the form of (\ref{ConRule}).}
\end{proof}

\subsection{Proof of Theorem \ref{CosTerm} } \label{ProofCosTerm}

\begin{proof}
For proof of (i), note the fact that if $h(x)$ is a continuous and piece-wise differential function over $[a,b]$ such that $h(a)=h(b)$ while $h(x)\neq h(a)$ for all $x$ in $(a,b)$, $(dh/dx)(a^+)$ and $(dh/dx)(b^-)$ cannot be negative simultaneously. From (\ref{CosDy}), we have $\mathbf{p}_i(T)=0$, which gives $\dot{p}_{ii}(T^-)=\lim_{t\rightarrow T^-}\dot{p}_{ii}(t)=-f'_i(x_i(T))<0$. Hence, there exists $\epsilon_i>0$ such that $p_{ii}>0$ and $p_{ij}\geq 0$ for $j\in\mathcal{N}/\{i\}$ over $[T-\epsilon_i,T]$. Suppose that one of the costate component $p_{ij}$ violates the inequality first at $t_a<T$, i.e., we have $p_{ij}(t)\geq 0$ for $t\geq t_a$, we obtain $\dot{p}_{ij}(t_a^+)\leq 0$ which is not feasible. If it is $p_{ii}$ that first violates the inequality at time $t_a$, we have $\dot{p}_{ii}(t_a^+)=-f'_i(x^*_i(t))-\sum_{j\in \mathcal{N}_{i,o}^{in}}p_{ij}(t)(1-x^*_j(t))u^*_{ji}(t)\beta_i<0$ which is in contradiction with the fact.

To prove (ii), note that at time $T$, we have $\mathbf{p}_i(T)=0$. Combined with the fact that $g_{ij}(\cdot)$ is convex and continuously differentiable, $g_{ij}(0)=0$ and $p_{ii}(t)> 0$, by expression (\ref{ConRule}), we have that $u^*_{ij}(t)=w^o_{ij}$ holds only at $t=T$.

For proof of (iii), from Observation \ref{CutDownSpace}, we know that if $w^o_{ji}=0$, the optimal weight adaptation $u^*_{ji}(t)=0$ for every $t$. Here, $|\mathcal{N}^{in}_{i,o}|=0$ indicates $w^o_{ji}=0$ for all $j\in\mathcal{N}$. Based on (\ref{CosDy}) and (\ref{CoMat}), the dynamics of the costate component $p_{ii}(t)$ can be written as
\begin{equation*}
    \dot{p}_{ii}(t)= \Large\left\{\sum\limits_{j\in\mathcal{N}_{i,o}^{out}}u^*_{ij}(t)\beta_jx^*_j (t)+\sigma_i \Large\right\} p_{ii}(t)-\alpha_i.
\end{equation*}
Note that the first term in the bracket is non-negative and $p_{ii}(T)=0$. Moving backward from $T$, it's obvious that $p_{ii}(t)$ is bounded above by $\alpha_i/\sigma_i$.

Under conditions stated in (iv), the dynamics of the costate component $p_{ii}(t)$ can be written as 
\begin{equation*}
    \dot{p}_{ii}(t)= -\alpha_i+\sum\limits_{j\in\mathcal{N}_{i,o}^{in}}-(1-x_j^*(t))u_{ji}^*(t)\beta_i p_{ij}(t).
\end{equation*}
We have proved in (i) that $p_{ij}(t) \geq 0$ holds for all $i,j\in\mathcal{N},i\neq j$ and for $t\in[0,T]$ which indicates $\dot{p}_{ii}<0$. Thus, $p_{ii}(t)$ is strictly monotonically decreasing over $t$.
\end{proof}

\subsection{Proof of Theorem \ref{SocNE}} \label{ProofSocNE}
\begin{proof}
Assume that $\mathbf{u}_o=(\mathbf{u}_1^o,...,\mathbf{u}_n^o)$ is an open-loop optimal control of the centralized control problem (\ref{OpCon}), and $\mathbf{x}_o=(\mathbf{x}_1^o,...,\mathbf{x}_n^o)$ is the state path under the optimal control. Fix an arbitrary $i\in \mathcal{N}$, and let $\mathbf{u}_i \neq \mathbf{u}_i^o$
be an open-loop strategy for player $i$. Let $\mathbf{x}=(\mathbf{x_1},...,\mathbf{x}_n)$ be the new state trajectory given by (\ref{OpCon}) corresponding to $(\mathbf{u}_i, \mathbf{u}_{-i}^o)$. As $\mathbf{u}^o$ and $\mathbf{x}^o$ are optimal for the optimal control problem (\ref{OpCon}), then
\begin{equation*}
\begin{aligned}
\int\limits_0^T \sum\limits_{i=1}^N f_i(x_i(t)){\color{blue}+}\sum\limits_{j\neq i} \sum\limits_{k\in \mathcal{N}_{j,o}^{out}}g_{jk}(u^o_{jk}(t)-w^o_{jk})\\ 
+ \sum\limits_{k\in \mathcal{N}_{i,o}^{out}}g_{ik}(u_{ik}(t)-w^o_{ik})dt\\
\geq \int\limits_0^T \sum\limits_{i=1}^N f_i(x^o_i(t)){\color{blue}+}\sum\limits_{i=1}^N \sum\limits_{j\in \mathcal{N}_{i,o}^{out}}g_{ij}(u^o_{ij}(t)-w^o_{ij})dt.\\
\end{aligned}
\end{equation*}
Adding to both sides of this this inequality the constant
\begin{equation*}
{\color{blue}-}\int\limits_0^T \sum\limits_{j\neq i} \sum\limits_{k\in \mathcal{N}_{j,o}^{out}}g_{jk}(u^o_{jk}(t)-w^o_{jk})dt,
\end{equation*}
we obtain that $\hat{J}_i (\mathbf{u}_i,\mathbf{u}_{-i}^o) \geq \hat{J}_i(\mathbf{u}^o)$ for all $\mathbf{u}_i \in U_i$. According to the definition of open-loop NE for differential games in (\ref{DefNE}), we know $\mathbf{u}^o$ is also an open-loop NE for the differential game with penalties.

To show the optimal control problem (\ref{OpCon}) shares the same necessary conditions with the new differential game, we again utilize the maximum principle. The Hamiltonian of player $i$ for the new differential game is $\hat{H}_i=H_i+c_i(t)$. We can find that relations (\ref{NecConDy}) (\ref{ConHam}) and (\ref{CosDy}) under the Hamiltonian $\hat{H}_i$ are aligned with relations (\ref{SoOp1}) (\ref{SoOp2}) and (\ref{SoOp3}) where $\lambda(t)=\mathbf{p}_i(t)$ at each $t$ for all $i\in\mathcal{N}$.
\end{proof}

\section{Example}\label{Corr2Examp}

To illustrate Corollary \ref{SocialSpe}, we consider a directed network in Fig. \ref{DAG}. Here, $\mathcal{R}_1 = \{1\}$, $\mathcal{R}_3=\{1,2,3,4\}$, $\mathcal{R}_5=\{1,4,5\}$. The $\Gamma(t)$ associated with this network can be rewritten as an upper triangular block matrix. The upper triangular matrix is denoted by $\Gamma_{\mathcal{R}_i}$ where the first $|\mathcal{R}_i|$ rows and columns of this matrix represent the vertices in $\mathcal{R}_i$ in an ascending order. The last $N-|\mathcal{R}_i|$ rows and columns represent the rest of the vertices in $\mathcal{N}\backslash \mathcal{R}_i$ in an ascending order. For example, the permutation for agent $5$ is $\{1,4,5,2,3\}$. Thus, the dynamics of $\mathbf{p}_5$ under the differential game given in Corollary \ref{SocialSpe} can be written as 
\[
\left[
\begin{array}{c}
\dot{p}_{51}       \\
    \dot{p}_{54}       \\
  \dot{p}_{55} \\\hline
    \dot{p}_{52}  \\
    \dot{p}_{53}
\end{array}
\right]
=\left[\begin{array}{@{}c|c@{}}
  \Gamma_{\mathcal{R}_5}^u &
 \mathbf{0}_{3\times2}
\\ \hline
  \begin{matrix}
  \Gamma_{\mathcal{R}_5}^l
  \end{matrix}
  & \Gamma_{\mathcal{R}_5}^r
\end{array}\right]
\left[
\begin{array}{c}
p_{51}       \\
p_{54}       \\
p_{55} \\\hline
p_{52}  \\
p_{53}
\end{array}
\right]+
\left[
\begin{array}{c}
-f_1'(x_1)       \\
-f_4'(x_4)       \\
-f_5'(x_5) \\\hline
0 \\
0
\end{array}
\right],
\]
where
\begin{equation*}
\begin{aligned}
\Gamma_{\mathcal{R}_5}^u&= \begin{bmatrix}
\sum\limits_{j\in \{2,3\}} u^*_{1j}\beta_jx^*_j+\sigma_1 & 0& 0\\
-(1-x_1^*)u_{14}^*\beta_4 & \sum\limits_{j\in \{3,5\}} u^*_{4j}\beta_j x^*_j+\sigma_4 & 0\\
 0&-(1-x_4^*)u_{45}^*\beta_5& \sigma_5
\end{bmatrix},\\
\Gamma_{\mathcal{R}_5}^l&=\begin{bmatrix}
-(1-x_1^*)u_{12}^*\beta_2 & 0 &0\\
-(1-x_1^*)u_{13}^*\beta_3 & -(1-x_4^*)u_{43}^*\beta_3&0\\
\end{bmatrix},\\
\Gamma_{\mathcal{R}_5}^r &=\begin{bmatrix}
u^*_{23}\beta_3 x^*_3 +\sigma_2 & 0\\
-(1-x_2^*)u_{23}^*\beta_3 &\sigma_3
\end{bmatrix}.
\end{aligned}
\end{equation*}
Thus, if we let $c_i(t)=\sum_{j\in\mathcal{R}_{i,o}\backslash \{i\}} f_j(x_j)$, the dynamics of $p_{ii}$ described by $(\ref{NewCosDy})$ is consistent with the dynamics of the $i$th component of $\lambda$ described by (\ref{SoOp2}). By solving the optimization problem (\ref{SoOp3}), we know that the optimal control problem shares the same control rule (\ref{ConRule}) with the differential game problem. Since $p_{ii} (t) = \lambda_i (t)$ for every $t\in [0,T]$, we can see the statement in Corollary \ref{SocialSpe} holds.


\ifCLASSOPTIONcaptionsoff
  \newpage
\fi

\end{document}